\documentclass[12pt,reqno]{amsart}
\usepackage{hyperref}
\usepackage{amssymb,amsthm}
\usepackage{amsfonts,cmtiup,comment,stmaryrd}

\usepackage{mathrsfs,cmtiup}

\makeatletter
\def\blfootnote{\xdef\@thefnmark{}\@footnotetext}
\makeatother

\newtheorem{theorem}{Theorem}[section]
\newtheorem{lemma}[theorem]{Lemma}

\theoremstyle{definition}

\newtheorem*{definition*}{Definition}
\newcommand{\ed}{\end{document}}

\renewcommand{\geq}{\geqslant}
\renewcommand{\leq}{\leqslant}

\let\leq=\leqslant
\let\geq=\geqslant
\numberwithin{equation}{section}

\begin{document}

\title[restricted centralizers of powers]{Profinite groups with restricted centralizers of powers}

\author{Cristina Acciarri}
\address{C.~Acciarri: Dipartimento di Scienze Fisiche, Informatiche e Matematiche, Universit\`a degli Studi di Modena e Reggio Emilia, Via Campi 213/b, I-41125 Modena, Italy}
\email{cristina.acciarri@unimore.it}

\author{Pavel Shumyatsky}

\address{P. Shumyatsky: Department of Mathematics, University of Brasilia, DF~70910-900, Brazil}
\email{pavel@unb.br}

\thanks{The first author is member of ``National Group for Algebraic and Geometric Structures, and Their Applications'' (GNSAGA–INdAM). The second author was partially  supported  by the Conselho Nacional de Desenvolvimento Cient\'{\i}fico e Tecnol\'ogico (CNPq),  and Funda\c c\~ao de Apoio \`a Pesquisa do Distrito Federal (FAPDF), Brazil.}


\begin{abstract}
 A group $G$ is said to have restricted centralizers if for every $x\in G$ the centralizer $C_G(x)$ either is finite or has finite index in $G$. Shalev showed that a profinite group with restricted centralizers is virtually abelian. Here we take interest in profinite groups $G$ for which there is an integer $n$ such that $C_G(x^n)$ is either finite or open whenever $x\in G$. It is shown that such a group $G$ has an open normal subgroup $T$ with the property that $G/Z(T)$ has finite exponent.
\end{abstract}

\maketitle

\section{Introduction}

A group $G$ is said to have restricted centralizers if for every $g\in G$ the centralizer $C_G(g)$ either is finite or has finite index in $G$. This notion was introduced by Shalev in \cite{shalev} where he showed that a profinite group with restricted centralizers is virtually abelian. We say that a profinite group has a property virtually if it has an open subgroup with that property. The article \cite{DMS_20} handles profinite groups with restricted centralizers of $w$-values for a multilinear commutator word $w$. The main result of \cite{DMS_20} says that the verbal subgroup $w(G)$ is virtually abelian. More recently, it was shown in \cite{accshu} that if $\pi$ is a set of primes and $G$ is a profinite group with restricted centralizers of $\pi$-elements, then $G$ has an open subgroup of the form $P\times Q$, where $P$ is an abelian pro-$\pi$ subgroup and $Q$ is a pro-$\pi'$ subgroup.

Here we take interest in profinite groups $G$ for which there is an integer $n$ such that $C_G(x^n)$ is either finite or open whenever $x\in G$. This case is really quite different from the aforementioned ones.

\begin{theorem}\label{main}
Let $n$ be a positive integer and $G$ a profinite group in which the centralizer of $x^n$ is either finite or open for every $x\in G$. Then $G$ has an open normal subgroup $T$ such that $G/Z(T)$ has finite exponent. 
\end{theorem}

Recall that a group $G$ is said to have finite exponent if there is a positive number $e$ such that $x^e=1$ for every $x\in G$.

The proof of the above theorem is short but it uses a number of highly nontrivial tools. In particular, it uses Zelmanov's theorem that a compact torsion group is locally finite \cite{ze}, which in turn depends on Wilson's reduction theorem \cite{wi} as well as on the classification of finite simple groups. An important role is also played by the theorem of Khukhro that a locally finite group admitting an automorphism of prime order with finite centralizer is virtually nilpotent \cite{khukhro_90}, \cite{khukhro_93}. Our proof of Theorem \ref{main} also depends on some recent results of probabilistic nature (see \cite{AzShu, AzShu2}).

Observe that Theorem \ref{main} provides a rather detailed information on the structure of profinite groups with restricted centralizers of $n$th powers. 

In particular, it shows that such a group $G$ has an abelian normal subgroup $N$ for which $G/N$ has finite exponent. 
Moreover, using a theorem of Mann \cite{Mann2007}, it is easy to deduce that the commutator subgroup $T'$ has finite exponent and therefore $G$ is (finite exponent)-by-abelian-by-finite.

\section{Proof of Theorem \ref{main}}

Throughout, by a subgroup of a topological group we mean closed subgroup. Whenever $X$ is a subset of a topological group, $\langle X\rangle$ stands for the subgroup generated by $X$. We write $\mu_G$ to denote the normalized Haar measure on a compact group $G$ (see \cite[Chapter 4]{hr_1}  or \cite[Chapter II]{nachbin}). If $K$ is a subgroup of a finite (or compact) group $G$, then  the probability that a random element of $G$ commutes with a random element of $K$ is denoted by $Pr(K,G)$. In \cite{AzShu} the following theorem was established. 

\begin{theorem}\label{virtually_central_by_torsion}Let $K$ be a subgroup of a compact group $G$. Then $Pr(\langle x\rangle, G)>0$ for any $x \in K$ if and only if $G$ has an open normal subgroup $T$ such that $K/C_K(T)$ is torsion.
\end{theorem}

Recall that the set of the elements of a group $G$ having centralizer of finite index is denoted by $FC(G)$.
The previous result of probabilistic nature is relevant to our study of profinite groups with restricted centralizers of $n$th powers of elements  in view of the following observation.

\begin{lemma}\label{rel_prob_positive}
Let $G$ be a compact group and $x\in G$. Then $Pr(\langle x\rangle, G)>0$ if and only if there is a positive integer $n$ such that $C_G(x^n)$ is open in $G$.
\end{lemma}
\begin{proof}  Suppose that for some positive integer $n$ the centralizer $C_G(x^n)$ is open in $G$, and let $k$ be the index $[G:C_G(x^n)]$. We have
 \begin{equation*}
    \begin{split}
Pr(\langle x\rangle, G)=& \int\limits_{G} \mu_{\langle x\rangle}(C_{\langle x\rangle}(g)) d\mu_G(g) = \int\limits_{G} \mu_{\langle x\rangle}(C_{\langle x\rangle}(g)\langle x^n\rangle) \mu_{\langle x^n\rangle}(C_{\langle x^n\rangle}(g))d\mu_G(g) \\
& \geq \int\limits_{G} \frac{1}{n}\mu_{\langle x^n\rangle}(C_{\langle x^n\rangle}(g))d\mu_G(g)=\frac{1}{n}Pr(\langle x^n\rangle, G)=\frac{1}{nk}>0.    
   \end{split}
 \end{equation*}

For the other implication, assume that $Pr(\langle x\rangle, G)>0$. By \cite[Proposition 1.1]{AzShu2} there are open subgroups $T\leq G$ and $B\leq \langle x\rangle$ such that $[B,T]$ is finite.  Since $B$ is open in $\langle x\rangle$, there is a positive integer $n$ such that  $x^n$ is a generator of  $B$. It follows from the finiteness of $[B,T]$ that $C_T(x^n)$ has finite index in $T$. Combining this  with the fact that $T$ is open in $G$, we deduce that the centralizer $C_G(x^n)$ has finite index in $G$ and so is open, as desired.
\end{proof}

It is a longstanding open problem, raised in 1970 by Hewitt and Ross \cite{hr}, whether every compact torsion group has finite exponent. In some cases the answer is known to be positive. We will need the following lemma. 

\begin{lemma}\label{first} Let $n$ be a positive integer and $G$ a  profinite torsion group in which $x^n\in FC(G)$ for any $x\in G$. Then $G$ has finite exponent.
\end{lemma}
\begin{proof} Let $K=G^n$ be the subgroup generated by the $n$th powers of elements of $G$. Since $G/K$ has finite exponent, it is sufficient to show that $K$ has finite exponent. For every positive integer $s$, let $X_s$ denote the set of all elements $g\in K$ such that $g^s=1$. Obviously the sets $X_s$ are closed in $G$ and the union of $X_s$ is exactly $K$. Therefore, by Baire Category Theorem \cite[Theorem 34]{kelley}, at least one of the sets $X_s$ contains non-empty interior. Hence, there is an integer $i$, an element $g\in K$ and an open normal subgroup $N$ of $G$ such that  $g(K\cap N)$ is contained in $X_i$ and so $(gx)^i=1$ whenever $x\in K\cap N$. Set $N_0=K\cap N$.

Thus, $N_0$ is a normal subgroup of $G$ having finite index in $K$. Every element of $K$ can be written as a product of (possibly infinitely many) $n$th powers of elements of $G$. Since $N_0$ has finite index in $K$, every element of $K$ can be written as a product of finitely many $n$th powers of elements of $G$ and an element of $N_0$. It follows that the coset $gN_0$ contains an element that is a a product of finitely many $n$th powers of elements of $G$. Without loss of generality we can assume that $g$ is such an element. Since all $n$th powers are contained in $FC(G)$, we deduce that $g\in FC(G)$.

In particular, $C_{N_0}(g)$ has finite index in $N_0$. Since $(gx)^i=1$ whenever $x\in N_0$, it follows that every element of $C_{N_0}(g)$ has order dividing $i$. In other words, the exponent of $C_{N_0}(g)$ divides $i$. Taking into account that $N_0$ has finite index in $K$ and $C_{N_0}(g)$ has finite index in $N_0$, we conclude that $K$ has finite exponent. This completes the proof.
\end{proof}

We are now ready to prove Theorem \ref{main}. 
\begin{proof}
Recall that $G$ is a profinite group in which the centralizer of $x^n$ is either finite or open for every $x\in G$. We need to show that $G$ has an open normal
subgroup $T$ such that $G/Z(T)$ has finite exponent. 

Suppose first that $x^n\in FC(G)$ for all $x\in G$.
In view of Lemma \ref{rel_prob_positive} this means that  $Pr(\langle x\rangle, G)>0$ for any $x\in G$. It follows from Theorem \ref{virtually_central_by_torsion} that $G$ has an open normal subgroup $T$ such that   $\bar{G}=G/C_G(T)$ is torsion. In view of Lemma \ref{first} the quotient $\bar{G}$ has finite exponent. Note that $Z(T)$ is open in $C_G(T)$ since $Z(T)=T\cap C_G(T)$. Therefore $G/Z(T)$ has finite exponent, as desired.

Thus, without loss of generality we assume that $G$ contains at least one element $x$ such that $C_G(x^n)$ is finite. If $G$ has an element  $g$ of infinite order, set $D=C_G(g^n)$. Observe that $D$ contains $\langle g^n\rangle$. Therefore $D$ is infinite and hence open in $G$. Moreover, for any $y\in D$ the centralizer $C_D(y)$ is infinite because it contains $\langle g^n\rangle$. It follows that $C_D(y^n)$ is open in $G$ and so $y^n\in FC(G)$ for any $y\in D$. In view of the above we conclude that $D$ has an open normal subgroup $K$ such that $D/Z(K)$ has finite exponent. Of course, $K$ can be chosen normal in $G$ and so  $G/Z(K)$ has finite exponent, as desired. 

Therefore we assume now that $G$ does not contain elements of infinite order, that is, $G$ is torsion. Moreover, by Zelmanov's theorem, $G$ is locally finite. Choose an $n$th power $a=b^n\in G$ of minimal possible order such that $C_G(a)$ is finite. For some prime $p$  the order of  $a^p$ is smaller than that of $a$ and so by the minimality of $|a|$ the centralizer $C_G(a^p)$ is open in $G$. Note that $a$ induces an automorphism of order $p$ on $C_G(a^p)$ and the centralizer $C_{C_G(a^p)}(a)$ is finite.  Khukhro's theorem \cite[5.4.1 Corollary]{khukhro_93} now tells us that $C_G(a^p)$ is virtually nilpotent. Thus, $G$ is virtually nilpotent. Without loss of generality we can assume that $G$ is nilpotent and infinite.  Since the centre of an infinite nilpotent profinite group is always  infinite (see for instance \cite[Lemma 2.1]{DMS_23}), we deduce that  $C_G(x)$ is infinite for any $x\in G$.  Thus $[G:C_G(x^n)]$ is finite for any  $x\in G$. Since $C_G(a)$ is finite, it follows that $G$ is finite, a contradiction. The proof is now complete.
\end{proof}

Remark that the proof shows that when $G$ is a virtually nilpotent profinite group containing a subset $X$ such that $C_G(x)$ is either finite or open for every $x\in X$, then  $X\subseteq FC(G)$.

\end{document}